\documentclass[12pt,reqno]{article}

\usepackage[usenames]{color}
\usepackage{amssymb}
\usepackage{enumerate}
\usepackage{array}
\usepackage{multicol}
\usepackage{multirow}
\usepackage{caption}

\usepackage[colorlinks=true,
linkcolor=webgreen,
filecolor=webbrown,
citecolor=webgreen]{hyperref}

\definecolor{webgreen}{rgb}{0,.5,0}
\definecolor{webbrown}{rgb}{.6,0,0}

\usepackage{color}
\usepackage{fullpage}
\usepackage{float}

\usepackage{amsmath,amssymb,epsfig,xspace}
\usepackage{amsthm}
\usepackage{amsfonts}
\usepackage{epsf}
\usepackage{pstricks}
\usepackage{pstricks-add}

\newcommand{\seqnum}[1]{\href{http://oeis.org/#1}{\underline{#1}}}

\begin{document}

\theoremstyle{plain}
\newtheorem{theorem}{Theorem}
\newtheorem{corollary}[theorem]{Corollary}
\newtheorem{lemma}[theorem]{Lemma}
\newtheorem{proposition}[theorem]{Proposition}

\theoremstyle{definition}
\newtheorem{definition}[theorem]{Definition}
\newtheorem{example}[theorem]{Example}
\newtheorem{conjecture}[theorem]{Conjecture}

\theoremstyle{remark}
\newtheorem{remark}[theorem]{Remark}

\newcommand\BD{\mathrm{B}}
\newcommand\SD{\mathrm{S}}
\newcommand{\N}{{\mathbb N}}
\newcommand{\Z}{{\mathbb Z}}
\newcommand{\C}{{\mathcal C}}
\newcommand{\A}{{\mathcal A}}
\newcommand{\F}{{\mathcal F}}

\begin{center}
\vskip 1cm{\LARGE\bf  Double-Recurrence Fibonacci Numbers and Generalizations\\
\vskip 1cm}
\large
Ana Paula Chaves\\
Instituto de Matem\'atica e Estat\'istica\\
Universidade Federal de Goi\'as\\
\href{mailto:apchaves@ufg.br}{\tt apchaves@ufg.br} \\
\ \\
Carlos Alirio Rico Acevedo\\
Departamento de Matem\'atica \\
Universidade de Bras\'ilia \\
\href{alirio@mat.unb.br}{\tt alirio@mat.unb.br}
\ \\
Brazil
\end{center}

\vskip .2 in

\begin{abstract}
Let $(F_n)_{n\geq 0}$ be the Fibonacci sequence given by the recurrence $F_{n+2}=F_{n+1}+F_n$, for $n\geq 0$, where $F_0=0$ and $F_1=1$. There are several generalizations of this sequence and also several interesting identities. In this paper, we investigate a homogeneous recurrence relation that, in a way, extends the linear recurrence of the Fibonacci sequence for two variables, called {\it double-recurrence Fibonacci numbers}, given by ${F(m,n) = F(m-1, n-1)+F (m-2, n-2)}$, for $n,m\geq 2$, where $F (m, 0) = F_m$, $F (m, 1) = F_{m+1}$, $F (0, n) = F_n$ and $F (1, n) = F_{n+1}$. We exhibit a formula to calculate the values of this double recurrence, only in terms of Fibonacci numbers, such as certain identities for their sums are outlined. Finally, a general case is  studied.
\end{abstract}

\section{Introduction}
Fibonacci numbers are known for their amazing properties, association with geometric figures, among others \cite{fib-pos, kalman}. Using the usual notation for such numbers, $(F_n)_{n\geq0}$, they are given by the following linear recurrence of order two: $F_{n+2}=F_{n+1}+F_n$, for $n\geq 0$, where $F_0=0$ and $F_1=1$. The Fibonacci sequence has been generalized in many ways, some by preserving the initial conditions, and others by studying high order recurrences with similar initial conditions \cite{miller,miles}. 

Our interest relies in a generalization that uses a recurrence for two indices (called a {\it double-recurrence}), such as the one studied by Hosoya \cite{hosoya}, who defined a set of integers $\{f_{m,n}\}$ satisfying:
\[
f_{m,n} = f_{m-1,n} + f_{m-2,n},
\]
\[
f_{m,n} = f_{m-1,n-1} + f_{m-2,n-2},
\]
for all $m\geq 2$, $m\geq n \geq 0$, where
\[
f_{0,0} = f_{1,0}=f_{1,1}=f_{2,1}=1 \ .
\]
Those numbers, when arranged triangularly, provide the famous {\it Fibonacci Triangle} (also known as  {\it Hosoya's Triangle}). One of our goals is to construct an analogue of the Fibonacci Triangle, studying a similar double-recurrence.
The set of numbers $\{ F(m,n) \}$, will be required to satisfy the following,
\begin{equation}\label{dfibo}
F(m,n)=F(m-1,n-1)+F(m-2,n-2),  \mbox{ for } \ m,n \geq 2,
\end{equation}
with initial values
\begin{center}
$\begin{array}{l l}
 F(m,0)=F_{m},    &  F(1,n)=F_{n+1},  \\
 F(m,1)=F_{m+1},   &  F(0,n)=F_{n} \  . 
\end{array} $
\end{center}

The initial conditions above, along with (\ref{dfibo}), are sufficient to calculate the value of $F(m,n)$ at each $(m,n) \in \mathbb{N}^2$. We call the values of the set $\{F(m,n)\}$,  {\it double-recurrence Fibonacci numbers}. Note that $F(m,n)$ is a symmetric function, since the initial conditions above and below the main diagonal are the same, and that $F(k,i) = F(k,k-i)$ for all $0 \leq i \leq \lfloor k/2 \rfloor $. Figure 1, displays a few values for $F(m,n)$, considering the bottom left corner as the origin $(0,0)$, and the $(m,n)$ coordinate having the value for $F(m,n)$.

Consider the value of the coordinate $(7,4)$, given by $F(7,4)=19$, and then draw a parallel to the antidiagonal from this point towards the axis, where the interactions begin with initial values $F(3,0)=F_3$ and $F(4,1)=F_5$. This means that, in order to determine $F(7,4)$, we only needed the pair $F_3$ and $F_5$, in other words, only Fibonacci numbers. The following proposition, asserts that this property is true for all $F(m,n)$, meaning that these values can be obtained using only Fibonacci numbers.

\begin{center}
\begin{pspicture*}(-0.5,-0.5)(8,8)
\psgrid[subgriddiv=1,griddots=7,gridlabels=0pt](0,0)(7,7)
\rput (0,7){13} \rput (1,7){21} \rput (2,7) {18} \rput (3,7)  {19} \rput (4,7)  {19} \rput (5,7)  {18} \rput (6,7)  {21} \rput (7,7)  {13}
\rput (0,6) {8} \rput (1,6) {13} \rput (2,6) {11} \rput (3,6)  {12} \rput (4,6)  {11} \rput (5,6)  {13} \rput (6,6)  {8} \rput (7,6)  {21}
\rput (0,5) {5} \rput (1,5)  {8} \rput (2,5)  {7} \rput (3,5)  {7} \rput (4,5)  {8} \rput (5,5)  {5} \rput (6,5)  {13}   \rput (7,5)  {18}
\rput (0,4) {3} \rput (1,4)  {5} \rput (2,4)  {4} \rput (3,4)  {5} \rput (4,4)  {3} \rput (5,4)  {8} \rput (6,4)  {11} \rput (7,4)  {19}
\rput (0,3) {2} \rput (1,3)  {3} \rput (2,3)  {3} \rput (3,3)  {2} \rput (4,3)  {5} \rput (5,3)  {7} \rput (6,3)  {12}  \rput (7,3)  {19}
\rput (0,2) {1} \rput (1,2)  {2} \rput (2,2)  {1} \rput (3,2)  {3} \rput (4,2)  {4} \rput (5,2) {7} \rput (6,2)  {11}   \rput (7,2)  {18}
\rput (0,1) {1} \rput (1,1)  {1} \rput (2,1)  {2} \rput (3,1)  {3} \rput (4,1)  {5} \rput (5,1)  {8} \rput (6,1)  {13}  \rput (7,1)  {21}
\rput (0,0) {0} \rput (1,0)  {1} \rput (2,0)  {1} \rput (3,0)  {2} \rput (4,0)  {3} \rput (5,0)  {5} \rput (6,0)  {8} \rput (7,0)  {13}
\end{pspicture*}
\captionof{figure}{Double-Fibonacci Numbers}
\end{center}

\begin{proposition}\label{fffibo}
Let \ $m,n  \in \mathbb{N}$, and $F(m,n)$ be a double-recurrence Fibonacci number, with  $k:= \min \{m,n\}$. Then,
\begin{equation}\label{ffibo}
F(m,n)=F_{k}F_{\vert m-n \vert +2}+F_{k-1}F_{\vert m-n \vert}.
\end{equation}
\end{proposition}
\begin{proof}
We proceed by the induction principle for two variables. It Is straightforward that $F(0,0) = F_0 = F_0F_2+F_{-1}F_0$. So, supposing that (\ref{ffibo}) holds for all $\ i \leq m$ and $j \leq n$, we have
\begin{align*}
F(m+1,n) &= F(m,n-1) + F(m-1,n-2) \\
 & =  F_{k'}F_{\vert m-n+1 \vert +2}+F_{k'-1}F_{\vert m-n+1 \vert}  + \ F_{k'-1}F_{\vert m-n+1 \vert +2}+F_{k'-2}F_{\vert m-n+1 \vert},
\end{align*}
where $k'= \min \{m,n-1\} \Rightarrow k'-1= \min \{m-1,n-2\}$. Therefore,
\[
F(m+1,n)  =  F_{k'+1}F_{\vert (m+1)-n \vert +2}+F_{k'}F_{\vert (m+1)-n\vert},
\]
and since $k'+1 = \min \{m+1,n\}$, the identity holds in this case. Analogously, following the same steps, the identity also holds for $F(m,n+1)$, which completes the proof.  
\end{proof}

In the homogeneous double-recurrence (\ref{dfibo}), one could replace the initial conditions by a general linear recurrence sequence of order two, or even arithmetic functions. In other words, we have the following:

\begin{definition} \label{spinfunctiondef}
Let $m,n \in \mathbb{N}$. The function $H(m,n)$ satisfying
 \begin{equation}
H(m,n)=H(m-1,n-1)+H(m-2,n-2) \; \; 
 \label{Hfibo}
 \end{equation}
for all $ m,n \geq 2$, where the following initial conditions are given
\begin{center}
$\begin{array}{ll}
H(m,0)=H_{1}(m), \; \; & \;  \; H(0,n)=H_{2}(n),\\
H(m,1)=H_{1}^{2}(m), \; \; & \; \;  H(1,n)=H_{2}^{1}(n),
\end{array}$
\end{center}
with $H_1$, $H_2$, $H_1^2$ and $H_2^1$ arithmetic functions, is called a \textnormal{double-recurrence function}.  If $H_1$, $H_2$, $H_1^2$ and $H_2^1$ are linear recurrence sequences of order two, the function satisfying (\ref{Hfibo}) is called a  \textnormal{spin Function.}
\end{definition}

In this way, double-recurrence Fibonacci numbers are values of a spin Function, such as every Fibonacci and Lucas numbers.
Now, let $H(m,n)$ be a spin Function, where 
\begin{eqnarray}
 \nonumber H(m,0)=H_{1}(m) & \mbox{ with } & H_{1}(0)=a \ \mbox{ and } \ H_{1}(1)=b, \\
 \nonumber H(m,1)=H_{1}^{2}(m) & \mbox{ with } & H_{1}^{2}(0)=d \ \mbox{ and } \ H_{1}^{2}(1)=c, \\
 H(0,n)=H_{2}(n) & \mbox{ with } & H_{2}(0)=a \ \mbox{ and } \ H_{2}(1)=b, \\
 \nonumber H(1,n)=H_{2}^{1}(n) & \mbox{ with } & H_{2}^{1}(0)=d \ \mbox{ and } \ H_{2}^{1}(1)=c,
\label{conind}
\end{eqnarray}
and if $m=n$, we have a linear recurrence sequence of order two, given by:
\begin{equation}\label{cond}   
H(m,m)=H_{1}^{1}(m), \mbox{ with } H_{1}^{1}(0)=a \ \mbox{ and } \ H_{1}^{1}(1)=c.
\end{equation}

The motivation for the term {\it spin function}, relies on the way that we can reach, from the initial terms, all pairs of $(m,n) \in \mathbb{N}^2$, where the function is evaluated, using every secondary diagonal on it, that we refer as {\it strings}. A graphical representation of it, can be seen next.

\begin{multicols}{2}
\begin{center}
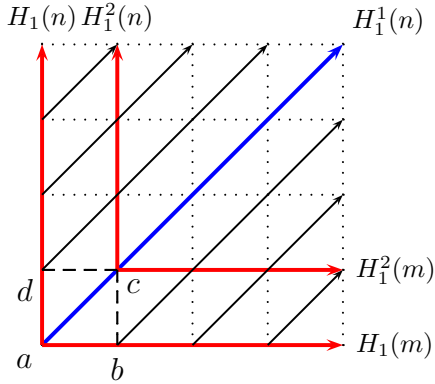

\begin{pspicture*}(-1,-1)(6,6)
\psgrid[subgriddiv=1,griddots=7,gridlabels=0pt](0,0)(4,4)
   \psline[linewidth=1.5pt,linecolor=blue]{->}(0,0)(4,4)
   \psline[linewidth=1.5pt,linecolor=red]{->}(0,0)(0,4)
   \psline[linewidth=1.5pt,linecolor=red]{->}(0,0)(4,0)
   \psline[linewidth=1.5pt,linecolor=red]{->}(1,1)(4,1)
   \psline[linewidth=1.5pt,linecolor=red]{->}(1,1)(1,4)
   \psline[linewidth=0.75pt]{->}(1,0)(4,3)
   \psline[linewidth=0.75pt]{->}(2,0)(4,2)
   \psline[linewidth=0.75pt]{->}(3,0)(4,1)
   \psline[linewidth=0.75pt]{->}(0,1)(3,4)
   \psline[linewidth=0.75pt]{->}(0,2)(2,4)
   \psline[linewidth=0.75pt]{->}(0,3)(1,4)
       \uput[0](4,0){\footnotesize $H_{1}(m)$}
       \uput[225](0,0){$a$}
       \uput[270](1,0){$b$}
   \psline[linestyle=dashed](0,1)(1,1)
       \uput[0](4,1){\footnotesize $H_{1}^{2}(m)$}
       \uput[-45](1,1){$c$}
       \uput[225](0,1){$d$}
   \psline[linestyle=dashed](1,0)(1,1)
       \uput[90](0,4){\footnotesize $H_{1}(n)$}
       \uput[90](1,4){\footnotesize $H_{1}^{2}(n)$}
       \uput[45](4,4){\footnotesize $H_{1}^{1}(n)$}
\end{pspicture*}
\captionof{figure}{A spin Function and its strings}
\end{center}

\begin{center}
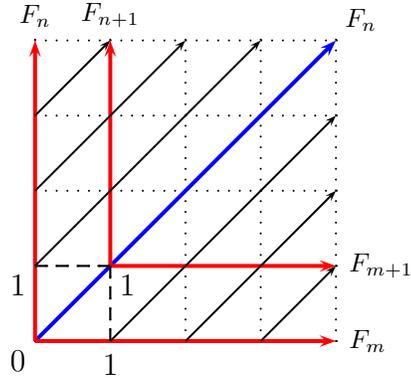

\begin{pspicture*}(-1,-1)(6,6)
\psgrid[subgriddiv=1,griddots=7,gridlabels=0pt](0,0)(4,4)
   \psline[linewidth=1.5pt,linecolor=blue]{->}(0,0)(4,4)
   \psline[linewidth=1.5pt,linecolor=red]{->}(0,0)(0,4)
   \psline[linewidth=1.5pt,linecolor=red]{->}(0,0)(4,0)
   \psline[linewidth=1.5pt,linecolor=red]{->}(1,1)(4,1)
   \psline[linewidth=1.5pt,linecolor=red]{->}(1,1)(1,4)
   \psline[linewidth=0.75pt]{->}(1,0)(4,3)
   \psline[linewidth=0.75pt]{->}(2,0)(4,2)
   \psline[linewidth=0.75pt]{->}(3,0)(4,1)
   \psline[linewidth=0.75pt]{->}(0,1)(3,4)
   \psline[linewidth=0.75pt]{->}(0,2)(2,4)
   \psline[linewidth=0.75pt]{->}(0,3)(1,4)
       \uput[0](4,0){\footnotesize $F_{m}$}
       \uput[225](0,0){$0$}
       \uput[270](1,0){$1$}
   \psline[linestyle=dashed](0,1)(1,1)
       \uput[0](4,1){\footnotesize $F_{m+1}$}
       \uput[225](0,1){$1$}
       \uput[-45](1,1){$1$}
   \psline[linestyle=dashed](1,0)(1,1)
       \uput[90](0,4){\footnotesize $F_{n}$}
       \uput[90](1,4){\footnotesize $F_{n+1}$}
       \uput[45](4,4){\footnotesize $F_{n}$}
  \end{pspicture*}
\captionof{figure}{Double-recurrence Fibonacci function}
\end{center}
\end{multicols}
\section{Properties and Identities}

Among several generalizations for Fibonacci numbers, we now consider the ones that satisfies the Fibonacci recurrence relation, but with arbitrary initial conditions.

\begin{definition}
Let $(G_n)_n$ a linear recurrence sequence of order two, where $G_1=a$, $G_2=b$ and $G_{n+2}=G_{n+1}+G_n$, $n \geq 1$. The ensuing sequence is called a \textnormal{generalized Fibonacci sequence (GFS)}.
\end{definition}

The following, is a classical result, that can be easily proved by induction, which states that every term on a GFS, can be written only in terms of Fibonacci numbers and their initial conditions.

\begin{theorem}\label{GFS}
Let $G_n$ denote the \textnormal{n}th term of the GFS. Then $G_{n+2}=bF_{n+1}+aF_{n}$, $n \geq 1$. 
\end{theorem}
\begin{proof}
 See \cite[Th. 7.1]{kos-f}.
\end{proof}

Note that, Proposition \ref{fffibo} can be seen as a generalization of Theorem \ref{GFS} for double-recurrence Fibonacci numbers. Our immediate purpose is to show that an analogous result also holds for spin functions. In order to do so, we introduce a double-recurrence function that will play the same role as Fibonacci numbers on Theorem \ref{GFS}. Let $m,n,a, b \in \mathbb{N}$. Then, define
\begin{equation}
F_{a}^{b}(m,n) := bF_{n}F_{\vert m-n \vert +2}+aF_{n-1}F_{\vert m-n \vert} \ .
\label{Fab}
\end{equation}

It is easy to see that $F_a^b(m,n)$ is a double-recurrence function, but not necessarily a spin Function, i.e., 
\[F_a^b(m+2,n+2) = F_a^b(m+1,n+1) + F_a^b(m,n), \] but the functions on the initial conditions are not necessarily linear recurrence sequences of order two. For that, we have the following result.

\begin{proposition}
Let $m,n \in \mathbb{N}$ and the spin function $H(m,n)$, such as on Definition \ref{spinfunctiondef}. Then,
\renewcommand{\labelenumi}{\roman{enumi}.}
\begin{enumerate}
\item{If $n \leq m-1$, then $H(m,n)=F_{a+b}^{c}(m-1,n)+F_{b}^{d}(m-2,n)$.}
\item{If $ m-1 < n$, then $H(m,n)=F_{a+d}^{c}(n-1,m)+F_{d}^{b}(n-2,m)$.}
\end{enumerate}
\label{HFiboab}
\end{proposition}
\begin{proof}

Let  $H(m, n)$ be a spin function for $n\leq m-1$, with functions $H_{1}^{2}$ and $H_{1}$ given by the initial conditions described previously. Similarly to the Proposition \ref{fffibo}, we have
\[H(m,n)=F_{n}H_{1}^{2}(m-n+1)+F_{n-1}H_{1}(m-n),\]
and since $H_{1}^{2}$ and $H_{1}$ are linear recurrence sequences, using 
Theorem \ref{GFS}, we get
\begin{align*}
H(m,n)& = F_{n}(cF_{m-n+1}+dF_{m-n})+F_{n-1}(bF_{m-n}+aF_{m-n-1})\\
& = cF_{n}F_{m-n+1}+aF_{n-1}F_{m-n-1}+dF_{n}F_{m-n}+bF_{n-1}F_{m-n}.
\end{align*}
Using that $bF_{n-1}F_{m-n}= b\cdot (F_{n-1}F_{m-n-2}+F_{n-1}F_{m-n-1})$, we obtain
\begin{align*}
H(m,n)& = cF_{n}F_{m-n+1}+(a+b)\cdot F_{n-1}F_{m-n-1}+dF_{n}F_{m-n}+bF_{n-1}F_{m-n-2}\\
& = F_{a+b}^{c}(m-1,n)+F_{b}^{d}(m-2,n) \ .
\end{align*}
Analogously, for $m-1<n$, considering $H_{2}^{1}$ and $H_{2}$, we get $$H(m,n)=F_{a+d}^{c}(n-1,m)+F_{b}^{d}(n-2,m),$$ which completes the proof. 
\end{proof}

\begin{multicols}{2}
\begin{center}
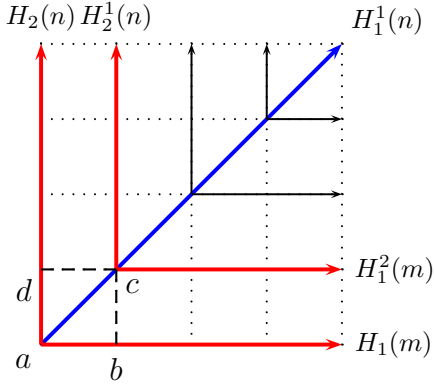

\begin{pspicture*}(-1,-0.5)(6,5)
\psgrid[subgriddiv=1,griddots=7,gridlabels=0pt](0,0)(4,4)
   \psline[linewidth=1.5pt,linecolor=blue]{->}(0,0)(4,4)
   \psline[linewidth=1.5pt,linecolor=red]{->}(0,0)(0,4)
   \psline[linewidth=1.5pt,linecolor=red]{->}(0,0)(4,0)
   \psline[linewidth=1.5pt,linecolor=red]{->}(1,1)(4,1)
   \psline[linewidth=1.5pt,linecolor=red]{->}(1,1)(1,4)
   \psline[linewidth=0.75pt]{->}(2,2)(4,2)
   \psline[linewidth=0.75pt]{->}(3,3)(4,3)
   \psline[linewidth=0.75pt]{->}(2,2)(2,4)
   \psline[linewidth=0.75pt]{->}(3,3)(3,4)
       \uput[0](4,0){\footnotesize $H_{1}(m)$}
       \uput[225](0,0){$a$}
       \uput[270](1,0){$b$}
  \psline[linestyle=dashed](0,1)(1,1)
       \uput[0](4,1){\footnotesize $H_{1}^2(m)$}
       \uput[225](0,1){$d$}
       \uput[-45](1,1){$c$}
  \psline[linestyle=dashed](1,0)(1,1)
      \uput[90](0,4){\footnotesize $H_{2}(n)$}
      \uput[90](1,4){\footnotesize $H_{2}^{1}(n)$}
      \uput[45](4,4){\footnotesize $H_{1}^{1}(n)$}
    \label{figura4}
\end{pspicture*}
\captionof{figure}{Graphical representation of Proposition \ref{HFiboab}}
\end{center}

\begin{center}
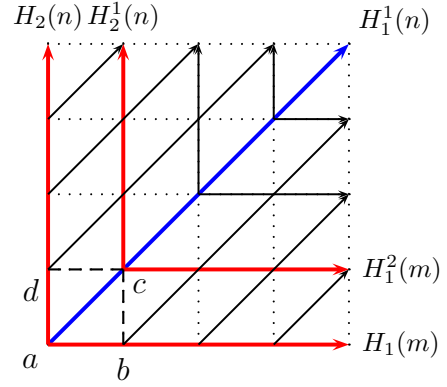

\begin{pspicture*}(-1,-0.5)(6,5)
\psgrid[subgriddiv=1,griddots=7,gridlabels=0pt](0,0)(4,4)
 \psline[linewidth=1.5pt,linecolor=blue]{->}(0,0)(4,4)
 \psline[linewidth=1.5pt,linecolor=red]{->}(0,0)(0,4)
 \psline[linewidth=1.5pt,linecolor=red]{->}(0,0)(4,0)
 \psline[linewidth=1.5pt,linecolor=red]{->}(1,1)(4,1)
 \psline[linewidth=1.5pt,linecolor=red]{->}(1,1)(1,4)
 \psline[linewidth=0.75pt]{->}(1,0)(4,3)
 \psline[linewidth=0.75pt]{->}(2,0)(4,2)
 \psline[linewidth=0.75pt]{->}(3,0)(4,1)
 \psline[linewidth=0.75pt]{->}(0,1)(3,4)
 \psline[linewidth=0.75pt]{->}(0,2)(2,4)
 \psline[linewidth=0.75pt]{->}(0,3)(1,4)
 \psline[linewidth=0.75pt]{->}(2,2)(4,2)
 \psline[linewidth=0.75pt]{->}(3,3)(4,3)
 \psline[linewidth=0.75pt]{->}(2,2)(2,4)
 \psline[linewidth=0.75pt]{->}(3,3)(3,4)
    \uput[0](4,0){\footnotesize $H_{1}(m)$}
    \uput[225](0,0){$a$}
    \uput[270](1,0){$b$}
 \psline[linestyle=dashed](0,1)(1,1)
    \uput[0](4,1){\footnotesize $H_{1}^2(m)$}
    \uput[225](0,1){$d$}
    \uput[-45](1,1){$c$}
 \psline[linestyle=dashed](1,0)(1,1)
    \uput[90](0,4){\footnotesize $H_{2}(n)$}
    \uput[90](1,4){\footnotesize $H_{2}^{1}(n)$}
    \uput[45](4,4){\footnotesize $H_{1}^{1}(n)$}
\end{pspicture*}
\captionof{figure}{Combination of Proposition \ref{HFiboab} and Definition \ref{spinfunctiondef}}
\end{center}
\end{multicols}

Now, we return our attention to sums of double-recurrence Fibonacci numbers. But first, we recall an interesting identity for Generalized Fibonacci Numbers \cite{wall}, giving an alternative proof for it.

\begin{proposition}
Let $(G_n)_n$ be a {\it GFS}, where $G_{n}=G_{n-1}+ G_{n-2}$ with initial conditions $G_{0}=g_{0} $ and $G_{1}=g_{1}$. Then 
\begin{equation}
\sum_{i=1}^{n}iG_{i}=nG_{n+2}-G_{n+3}+G_{3}
\end{equation}
\label{sumg}
\end{proposition}
\begin{proof}
Straightforward from Theorem \ref{GFS}, we have $G_{n}=g_{0}F_{n-1}+g_{1}F_{n}$. Thus,
\begin{align} 
\sum_{i=1}^{n}iG_{i} & = g_{0}\sum_{i=1}^{n}iF_{i-1}+g_{1}\sum_{i=1}^{n}iF_{i}  \nonumber\\ 
                     & = g_{0}\sum_{i=0}^{n-1}(i+1)F_{i}+g_{1}\sum_{i=1}^{n}iF_{i} \label{eq1} \\
                     & = g_{0}( (n-1)F_{n+1}-F_{n+2}+2+F_{n+1}-1) + g_{1}(nF_{n+2}-F_{n+3}+2) \label{eq2}  \\ 
                     & =  n(g_{0}F_{n+1}+g_{1}F_{n+2})-(g_{0}F_{n+2}+g_{1}F_{n+3})+2g_{0}+g_{1}  \nonumber\\
                     & =  nG_{n+2}-G_{n+3}+G_{3}
                     \nonumber
\end{align}
Where, from (\ref{eq1}) to (\ref{eq2}), the identity  $\sum_{i=1}^{n}iF_{i}=nF_{n+2}-F_{n+3}+2$,  \cite[p.16, Ex.10]{vorobiov}, is used.
\end{proof}

The following proposition, consists of a closed form to calculate the sums of Double-Fibonacci numbers, where the indices are in $\{1, \ldots , m\}^2$. 

\begin{proposition} \label{sumdoublefib}
Let $F(i,j)$ be Double-Fibonacci Numbers, where  $i,j \in  \{0,1,\ldots ,m\}$. Then,
\renewcommand{\labelenumi}{\roman{enumi}.}
\begin{enumerate}
\item{ The sum of all Double-Fibonacci Numbers with indices below the main diagonal, including it, is given by
\renewcommand{\thefootnote}{\fnsymbol{footnote}}
\begin{equation} \label{eqsum1}
\sum_{i,j=0\atop j\leq i}^{m}F(i,j)=\frac{2}{5}\left( mL_{m+3}-L_{m+4}+2F_{m+2} \right) +2. 
\end{equation}}

\item{The sum of all Double-Fibonacci Numbers, with indices on the square $m \times m$, is
\begin{equation*}
\sum_{i,j=0}^{m}F(i,j)=\frac{4}{5}\left( mL_{m+3}-L_{m+4}+2F_{m+2} \right)-F_{m+2} +5. 
\end{equation*}}
\end{enumerate}
\end{proposition}
\begin{proof}
First, we proceed to prove (i), and use it to prove (ii). Rewriting (\ref{eqsum1}), and using the closed form on Proposition \ref{fffibo}, we have
\begin{align*}
\sum_{i,j=0\atop i\geq j}^{m}F(i,j)&=\sum_{i=0}^{m}\sum_{j=0}^{i}F_{j}F_{i-j+2}+F_{j-1}F_{i-j}\\
                                   &=\sum_{i=0}^{m}\sum_{j=0}^{i} F_{j}F_{i-j+1}+F_{j-1}F_{i-j}+F_{j}F_{i-j}, \\
\intertext{and since $F_{i}=F_{j}F_{i-j+1}+F_{j-1}F_{i-j}$, it follows,} 
                                   &=\sum_{i=0}^{m}\sum_{j=0}^{i} F_{i}+F_{j}F_{i-j}\\
                                   &=\sum_{i=0}^{m}\left(\left(i+1 \right)F_{i}+ \sum_{j=0}^{i}F_{i-j}F_{j} \right) . \\
\intertext{Now, we observe that the sum $\sum_{j=0}^{i}F_{i-j}F_{j}$, is referenced as sequence \seqnum{A001629} on \cite{oeis}, where is established that it is equal to  $((i-1)F_{i}+2iF_{i-1})/5=(iL_{i}-F_{i})/5$, $(L_n)_{n\geq 0}$ being the Lucas Sequence, and the last equality follows from \cite[Eq.\ 32.13, p.\ 375]{kos-f}. Thus,}
\sum_{i,j=0\atop i\geq j}^{m}F(i,j) &=\sum_{i=0}^{m}(i+1)F_{i}+\sum_{i=0}^{m}\frac{iL_{i}-F_{i}}{5} . \\
\intertext{From Proposition \ref{sumg} and  $\sum_{i=1}^{n}F_{i}=F_{n+2}-1$, we have, }
                                   &=m\left(\frac{L_{m+2}}{5}+F_{m+2}\right) -\left(\frac{L_{m+3}}{5}+F_{m+3}\right)  +\frac{4}{5}F_{m+2}+2,
\intertext{then, finally by $L_{n-1}+L_{n+1}=5F_{n}$ (see \cite[Cor.\ 5.5, p.\ 80]{kos-f}), it follows that} 
                                   &=\frac{m\left( L_{m+2}+L_{m+1}+L_{m+3}\right)}{5} -\frac{\left( L_{m+4}+L_{m+2}+L_{m+3}\right)}{5}\\
                                  & \quad + \ \frac{4}{5}F_{m+2}+2 \\
\therefore \  \sum_{i,j=0\atop i\geq j}^{m}F(i,j)&=\frac{2}{5}\left( mL_{m+3}-L_{m+4}+2F_{m+2} \right) +2,
\end{align*}
completing the proof for (i). For (ii), we use the symmetry satisfied by double-recurrence Fibonacci Numbers, $F(m,n)=F(n,m)$, giving us that the sum on (ii) is two times the sum on (i), minus the sum for indices on the main diagonal:
\begin{align*}
\sum_{i,j=0}^{m}F(i,j)&=2\sum_{i,j=0\atop i\leq j}^{m}F(i,j)-\sum_{i=0}^{m}F(i,i)\\
                      &=\frac{4}{5}\left( mL_{m+3}-L_{m+4}+2F_{m+2} \right) +4 -\sum_{i=0}^{m}F_{i}\\
                      &=\frac{4}{5}\left( mL_{m+3}-L_{m+4}+2F_{m+2} \right)-F_{m+2}+5 .
\end{align*}
\end{proof}

Out of curiosity, equation (\ref{eqsum1}) happens to be the same formula for the path length of the Fibonacci tree of order $n$. (\seqnum{A178523} of \cite{oeis})

\section{Acknowledgements}
During the preparation of this paper, Ana Paula Chaves was  supported in part by CNPq Universal 01/2016 - 427722/2016-0 grant, and Carlos Alirio Rico Acevedo was fully supported by a Masters Scholarship from CNPq.

\bibliographystyle{amsplain}
\bibliography{Chaves-Acevedobi}

\bigskip
\hrule
\bigskip

\noindent 2010 {\it Mathematics Subject Classification}:
Primary 11B39; Secondary 11J86.

\noindent \emph{Keywords:}
Fibonacci numbers, double-recurrence sequence, closed form.

\bigskip
\hrule
\bigskip

\noindent (Concerned with sequences \seqnum{A001629}, \seqnum{A002940}, \seqnum{A006478}, \seqnum{A010049},  \seqnum{A014286}, \seqnum{A122491},  \seqnum{A178523}, \seqnum{A190062},)

\bigskip
\hrule
\bigskip

\section*{Appendix}
The following table explicit some interesting sequences founded on \cite{oeis}, that can be obtained from the sum of the terms of $H(i,j)$, with initial conditions $a,b,c$ and $d$, considering \linebreak $0 \leq j < i \leq n$, $0 \leq i \leq j \leq n$, and all $i,j \in \{0,1, \ldots , n\}^2$.

\begin{table}
\centering
 \begin{tabular}{|c|c|c|c|}
 \hline
\multicolumn{1}{|p{2cm}|}{\centering ~ \\ Initial Condition $[a,b,c,d]$}& \multicolumn{1}{p{3.3cm}|}{$$\sum_{i=1}^{n}\sum_{j=0}^{i-1}H(i,j)$$}
&\multicolumn{1}{p{3.3cm}|}{$$\sum_{j=0}^{n}\sum_{i=0}^{j}H(i,j)$$} &\multicolumn{1}{p{3.3cm}|}{$$\sum_{i,j=0}^{m}H(i,j)$$}\tabularnewline
 \hline
 $[0,0,0,1]$ &\seqnum{A006478}$(n)$&\seqnum{A001629}$(n+1)$&\seqnum{A006478}$(n+1)$\\
  \hline
 $[0,0,1,0]$ &\seqnum{A002940}$(n-2)$&\seqnum{A006478}$(n+1)$& -  \\
  \hline
 $[0,0,1,1]$ & - & \seqnum{A122491}$(n+2)$& - \\
  \hline
 $[0,1,0,0]$ & \seqnum{A001629}$(n+1)$ & \seqnum{A006478}$(n)$& \seqnum{A006478}$(n+1)$ \\
 \hline
  $[0,1,0,1]$ & \seqnum{A006478}$(n+1)$ & \seqnum{A006478}$(n+1)$& \seqnum{A178523}$(n+1)$ \\
   \hline
  $[0,1,1,0]$ &\seqnum{A014286}$(n)$ &  \seqnum{A002940}$(n-1)$& - \\
\hline
  $[0,1,1,1]$ & - & \seqnum{A178523}$(n+1)$ & -  \\
  \hline
  $[1,0,0,0]$ & \seqnum{A001629}$(n)$ & \seqnum{A010049}$(n+1)$ & - \\
   \hline
  $[1,0,0,1]$ &\seqnum{A122491}$(n+1)$ & \seqnum{A001629}$(n+2)$ & - \\
   \hline
  $[1,0,1,0]$ &\seqnum{A178523}$(n)$ & - & - \\
   \hline
  $[1,0,1,1]$ & - &\seqnum{A006478}$(n+2)$& - \\
  \hline
  $[1,1,0,0]$ &\seqnum{A006478}$(n+1)$&\seqnum{A190062}$(n+1)$ & -  \\
   \hline
  $[1,1,1,0]$ &\seqnum{A002940}$(n-1)$ & - & -  \\
  \hline
  $[1,1,1,1]$ & - &\seqnum{A014286}$(n+11)$ & -   \\
  \hline
 \end{tabular}
 \caption{Related sequences.}
 \end{table}

\end{document}